\documentclass[12pt,twoside,reqno]{amsart}
\usepackage{amsmath, amsfonts, amsthm, amssymb,amscd, graphicx, amscd}
\usepackage{float,epsf}
\usepackage[english]{babel}
\usepackage{enumerate}
\usepackage{tikz}

\usepackage{srcltx}
\usepackage{geometry}
\usepackage{verbatim}
\usepackage{mathrsfs}
\usepackage{hyperref}
\usepackage{enumitem} 
\usepackage{esint}

\newcommand{\R}{\mathbb{R}}

\newcommand{\tn}{\textnormal}

\newcommand{\eps}{\varepsilon}

\newtheorem{theorem}{Theorem}[section]
\newtheorem{lemma}[theorem]{Lemma}

\newtheorem{definition}[theorem]{Definition}

\newtheorem{remark}[theorem]{Remark}
\newtheorem{example}[theorem]{Example}
\newtheorem{notation}[theorem]{Notation}

\numberwithin{equation}{section}
\numberwithin{figure}{section}

\newcommand{\norm}[1]{\left\|#1\right\|}

\newcommand{\FF}{{\boldsymbol F}}

\renewcommand{\div}{\mathrm{div}} 

\renewcommand{\div}{\text{\sl div}\,} 

\def\intave#1{\int_{#1}\hbox{\llap{$\raise2.3pt\hbox{\vrule
height.9pt width7pt}\phantom{\scriptstyle{#1}}\mkern-2mu$}}}

\newcommand{\sizeofboundary}{R}

\begin{document}
\title[Group Theoretic Constructions of Singular Set in a Long Range Segregation Model]{Group Theoretic Constructions of Singular Set in a Long Range Segregation Model}

\author{Howen Chuah}
\address{Howen Chuah,
Department of Mathematics, Purdue University,
 150 N. University Street,
 West Lafayette, IN 47907-2067, USA}
\email{hchuah@purdue.edu}




\begin{abstract}
In this paper, we construct several explicit examples of singular sets of Hausdorff dimension $(n-2)$ in $\mathbb{R}^n$ on free boundaries for an elliptic system modeling long range segregation. The system has been previously studied by Caffarelli, Patrizi and Quitalo in \cite{CL2} for the regularity of the free boundary in dimension two, and by the author and Torres in \cite{ChPaTo26_2} for the partial regularity in higher dimensions. However, the dimension of the singular set is unknown, and no concrete examples of singular set are known in the literature due to the nonlocal nature of the elliptic system. In this paper, we overcome this difficulty by rigidity and finite group action. As a byproduct of our result, we see that singular points can exist for the model in any dimensions. We also show that our method can be applied to the study of the singular set in the adjacent model. Finally, we also discuss some related open problems for future studies.
\,\\\end{abstract}

\maketitle

\section{Introduction}

In this paper, we consider the following elliptic system 
\begin{equation}
\label{main_problems}
\begin{cases} \Delta u^{\varepsilon}_{i}= \frac{1}{\eps^2} u^{\eps}_i  \sum_{j \neq i}  H_R(u^{\eps}_j)(x)\quad & \text{ in } \Omega,\\
        u^{\eps}_i= f_i & \text{ on } (\partial \Omega)_{\leq {\sizeofboundary}}, \\
        
        u^{\eps}_i \geq 0 & \text{ in } \Omega \cup (\partial \Omega)_{\leq \sizeofboundary},
       \end{cases}
\end{equation}
for  $i=1,\ldots, K$, 
where $\Omega$ is a bounded Lipschitz domain in  $\R^n$, $\eps >0$, $0 < \sizeofboundary \leq 1$. The boundary neighborhood is defined as 
\begin{equation*}
    (\partial \Omega)_{\leq \sizeofboundary}:= \{x \in \Omega^c: d(x, \partial \Omega)\leq \sizeofboundary\} ,
\end{equation*}
where $d(x,\partial \Omega) := \inf_{y \in \partial \Omega}|x-y|$ denotes the Euclidean distance from $x$ to $\partial \Omega$. 
The boundary data are nonnegative H{\"o}lder continuous functions
with supports separated by at least distance $\sizeofboundary$ (see assumptions \eqref{fiassumption}-\eqref{f_i-f_j_disjointsupport}).

Each equation in the system is coupled to the others through a zero-order interaction term
$H_R(u_j^\eps)$, which depends  on the parameter $R$. 
In \eqref{main_problems}, we consider the case where $H_{\sizeofboundary}$ is defined by the following integral average:

\begin{equation}
\label{H1}
H_R(w)(x)= \fint_{{B}_\sizeofboundary(x)}w(y)\,dy,
\end{equation}

where $w \geq 0$. We assume that the boundary data satisfy  
\begin{equation}    
\label{fiassumption}
        f_i : (\partial \Omega)_{\leq R} \rightarrow \mathbb{R}, f_i \geq 0, f_i \neq 0, \quad \textnormal{$f_i$ is H\"older continuous,} 
    \end{equation}
and that there is a constant $c > 0$ such that, for any $x \in \partial \Omega \cap \text{supp } f_i$,
\begin{equation}
\label{size_of_ball_intersect_support}
|{B}_r(x) \cap \text {supp }f_i| \geq c|{B}_r(x)|,
\end{equation}
and
       \begin{equation}\label{f_i-f_j_disjointsupport}       
        d(\tn{supp} \, f_i, \tn{supp} \, f_j) \geq \sizeofboundary \textnormal{ for all } i \neq j. \\
        \end{equation}

The existence of positive solutions $(u_1^\eps,\ldots,u_K^\eps)$ of the system \eqref{main_problems} was proved in \cite{CL2}. Uniqueness of solutions to system \eqref{main_problems} was proved in \cite{Bozorgnia}. It was also shown in \cite{CL2} that solutions converge to a limit configuration $(u_1,\ldots, u_K)$ as $\eps \to 0$, where the supports of the populations $u_i$  are mutually disjoint and separated from each other by distance $\sizeofboundary$ (see Figure \ref{generalpicture}). The regularity of the free boundary for $n=2$ was also established. For system \eqref{main_problems}, one of the main challenges in the analysis of the free boundaries is that classical techniques (i.e., monotonicity formulas) can not be used. The techniques developed in \cite{CL2} rely on the concepts of {\it asymptotic cone} and {\it asymptotic angle}. At a singular point, the angle is measured as the intersection of asymptotic cones.  
This construction was generalized in \cite{ChPaTo26_2} to characterize the regular and singular points in any dimension in terms of angles and densities.

For the case of $n=2$, it was shown in \cite{CL2} that the singular points are isolated and the regular set is locally $C^1$. It was also shown, under additional conditions, that the free boundary is Lipschitz. A result on the equality of angles is also established. Another tool for $n=2$ is the construction of harmonic functions (i.e., barrier functions) on cones vanishing on the boundary. 

For higher dimensions, it was shown in \cite{ChPaTo26_2} that if the angles at the singular points are bounded away from $\frac{n\omega_n}{2}$, the regular set is open and locally $C^1$. It was also shown that under a convexity condition, the free boundary consist of finitely many hyperplanes of dimension $n-1$, and the angles of the singular points are bounded above by $\frac{n\omega_n}{3}$. Moreover, if there are two free boundary points that are at distance $\sizeofboundary$ apart form each other, either they are both regular or they are both singular. A symmetry condition on the free boundary is also derived.

System \eqref{main_problems} is an example of \eqref{modelproblem}, the Gause-Lotka-Volterra system in population dynamics, that models coexistence of species that live in the same territory, diffuse, and compete for limited resources.
 
\begin{equation}\label{modelproblem}L_i(u^\eps _i)=\frac{u^\eps_i}{\eps^2}F(u^\eps_1,\ldots,u^\eps_K),\end{equation}
in some domain $\Omega$, where $u^\eps_i$ is a positive function representing the density of the $i$-th species, $L_i$ encodes  the diffusion of  $u^\eps_i$, and   $u_i^\eps F(u^\eps_1,\ldots,u^\eps_K)/\eps^2$ models  the  attrition of the species $i$ due to competition with the others. The interaction functional  $F$ is strictly positive whenever the supports of two or more species overlap.
The smaller the parameter $\eps$, the stronger the competition among species. 
In the limit as $\eps\to0^+$ the high  competition  forces the species  to segregate, meaning $u_iu_j=0$ for $j\neq i$.

Another example of system \eqref{modelproblem} is given by 
\begin{equation}
\label{adjacentsegregation}
\begin{cases} \Delta u^{\varepsilon}_{i}= \frac{1}{\eps^2}   \sum_{j \neq i}  u^{\eps}_iu^{\eps}_j\quad & \text{ in } \Omega,\\
        u^{\eps}_i= f_i & \text{ on } \partial \Omega, \\
        u^{\eps}_i \geq 0 & \text{ in } \overline{\Omega}.
\end{cases}
\end{equation}
The existence of positive solutions to \eqref{adjacentsegregation} was initially investigated by Dancer and Du \cite{DancerDu1, DancerDu2} in the case of three species. 
Convergence to a segregated limit configuration as  $\eps\to0^+$ was later proven by Dancer, Hilhorst, Mimura, and Peletier \cite{DancerHilMimPel}. 
 More general classes of linear competitive systems, including \eqref{adjacentsegregation} as a special case, have been studied by Conti, Terracini, and Verzini \cite{Conti4, Conti3, Conti5}. %
 We also refer to \cite{CL7, Conti6} for related optimal partition problems involving the first eigenvalue of the Laplace operator.  
 The geometric properties of the free boundaries $\partial\{u_i>0\}\cap\Omega$ for the system \eqref{adjacentsegregation} have been investigated by Caffarelli, Karakhanyan and Lin \cite{CL5} (see also \cite{CL4}). It was shown that the free boundary splits into two parts: a regular set,  which is a locally  analytic surface, and a singular set, which is a closed set of Hausdorff dimension at most $n-2$. Singular points occur where the boundaries of three or more connected components of the supports intersect. See \cite{TavarezTerracini} for similar results applied to a broader class of systems.

{

The system \eqref{adjacentsegregation}, when the Laplace operator is replaced by the fully nonlinear negative Pucci operator, has been studied by Quitalo in \cite{V} and Caffarelli, Quitalo, Patrizi, and Torres in {\cite{CL3}}. 


The interaction between the populations in \eqref{adjacentsegregation} is local, meaning that it depends only on the value of the densities, $u_i(x)$, at the point $x$. The segregation is adjacent since the supports of the populations have a common boundary. However, there are many processes where the growth of species $i$ is inhibited by populations $j$ occupying an entire neighborhood around $x$, see for example \cite{CuMaMa, MiEiFang}. Caffarelli, Patrizi, and Quitalo \cite{CL2} introduced system \eqref{main_problems} with the Laplace operator as an example to model non-local interactions. The same system was studied by the author, Patrizi, and Torres in \cite{ChPaTo26} with Laplacian replaced by the negative Pucci operator.  
When $H$ is given by \eqref{H1} with $w^2(y)$ in place of $w(y)$, minimizing solutions and the limiting configurations of \eqref{main_problems} have been studied in \cite{NS1,NS2}.

However, no concrete examples of the free boundary were constructed in the literature for the long range segregation model, due to the nonlocal nature of the elliptic system \eqref{main_problems}. In this paper, we construct several explicit examples of singular sets based on rigidity and group action. We will construct for any $n \geq 2$ and $K \geq 3$ populations, a bounded domain $\Omega \subset \mathbb{R}^n$ and $K$ boundary data, such that the singular set on the free boundary is nonempty and has Hausdorff dimension exactly $n-2$. In the special case where $K = 2^n$, we can construct an example of $K$ boundary data on the unit ball $B_1(0) \subset \mathbb{R}^n$ such that the singular set has Hausdorff dimension $n-2$ and with a stratification result. Finally, we also give an example of two populations in any dimension in which the free boundaries are concentric spheres. 

We state our main result as follows.

\begin{theorem}
\label{Main_Result_1}
Consider the elliptic system \eqref{main_problems} with $H_R(w)(x)$ given by \eqref{H1}.  
\begin{enumerate}
\item For any $n \geq 2$ and $K \geq 3$, there exists a bounded Lipschitz domain $\Omega \subset \mathbb{R}^n$ and boundary data $f_1, \cdots, f_K : (\partial \Omega)_{\leq \sizeofboundary} \rightarrow \mathbb{R}$ satisfying \eqref{fiassumption}-\eqref{f_i-f_j_disjointsupport}, such that the singular set on the free boundary is nonempty with Hausdorff dimension exactly $n-2$.
\item Take $\Omega := B_1(0) \subset \mathbb{R}^n$ to be the unit open ball in $\mathbb{R}^n$. Then there exists $K = 2^n$ boundary data $f_1, \cdots, f_K: (\partial \Omega)_{\leq \sizeofboundary} \rightarrow \mathbb{R}$ satisfying \eqref{fiassumption}-\eqref{f_i-f_j_disjointsupport}, such that for each $i = 1, \cdots, K$, the singular set on the free boundary takes the form 
\begin{eqnarray}
\Gamma^{\textnormal{sin}}_{i,0} \cup \Gamma^{\textnormal{sin}}_{i,1} \cup \cdots \cup \Gamma^{\textnormal{sin}}_{i,n-2},
\end{eqnarray}
where for each $k \in \{0,1,\cdots,n-2\}$, $\Gamma^{\textnormal{sin}}_{i,k}$ is nonempty and is a union of finitely many $k$-dimensional manifolds.
\end{enumerate}
\end{theorem}

As an immediate consequence of Theorem \ref{Main_Result_1}, for any $n \geq 2$ there exists a bounded Lipschitz domain $\Omega$ and boundary data $f_1, \cdots, f_K : (\partial \Omega)_{\leq \sizeofboundary} \rightarrow \mathbb{R}$ such that the singular set is nonempty. This exhibits a different phenomenon with the Bernoulli free boundary problem, in which it is known that there is a critical dimension $d^* \in \{5,6,7\}$ such that singular points exist only for $n \geq d^*$. 


The paper is organized as follows. In section 2 we recall some known results and preliminaries. In section 3 we apply the method of group action to prove Theorem \ref{Main_Result_1}, and we also discuss some consequences of Theorem \ref{Main_Result_1} and related open problems.

\usetikzlibrary{calc}
\begin{figure}
\begin{tikzpicture}[scale=0.7]
    

    \filldraw[fill=blue!10, draw=blue!80, thick] 
        plot[domain=-4.6:4.6, samples=100] ({0.4*sin(1.5*\x r) - 0.5}, \x) 
        -- (-3.5, 3.1) arc (140:220:4.8) -- cycle;
    \node[blue!80, font=\large] at (-2.2, 0) {$S_1$};

    \filldraw[fill=red!10, draw=red!80, thick] 
        plot[domain=-4.6:4.6, samples=100] ({0.4*sin(1.5*\x r) + 0.5}, \x) 
        -- (3.5, 3.1) arc (40:-40:4.8) -- cycle;
    \node[red!80, font=\large] at (2.2, 0) {$S_2$};

    
\end{tikzpicture}
\caption{Possible configuration with two populations. The supports of the populations are the sets $S_i = \{u_i > 0\} \cap \Omega$. The distance between the the supports is $\sizeofboundary$.}
\label{generalpicture}
\end{figure}
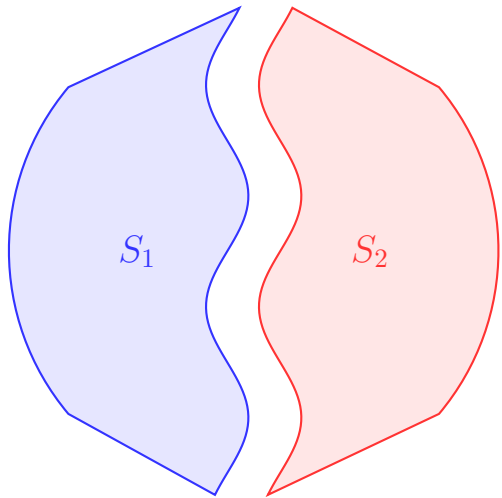

\section{Preliminaries}
We first recall some known definitions and results in \cite{Bozorgnia}, \cite{CL2}, \cite{ChPaTo26}, and \cite{ChPaTo26_2} for the convenience of the reader. 

Firstly, we have the existence and uniqueness of solutions of \eqref{main_problems}. The existence of solutions was proved in \cite[Thm 4.1]{CL2} using Perron's method and Schauder fixed point theorem. The similar approach was adopted in \cite[Thm 1]{ChPaTo26} to extend the result with negative Pucci operator as diffusion. The uniqueness result was proved in \cite[Thm 3.3]{Bozorgnia} by sub- and sup-solution method. 

\begin{theorem}
\label{existence_and_uniqueness}
\cite[Thm 3.3]{Bozorgnia}\cite[Thm 4.1]{CL2}\cite[Thm 1]{ChPaTo26}

Let $\Omega \subset \mathbb{R}^n$ be a bounded Lipschitz domain, and assume \eqref{fiassumption}-\eqref{f_i-f_j_disjointsupport} hold true. Consider the elliptic system \eqref{main_problems} with $H_R(w)(x)$ given by \eqref{H1}. Then for any $\eps > 0$, and $0 < \sizeofboundary \leq 1$, there exists a unique positive solution 
$(u^{\eps}_1,\ldots, u^{\eps}_K) \in  C^{\alpha}(\overline{\Omega};\mathbb{R}^K)\cap C_{loc}^{2,\alpha}(\Omega;\mathbb{R}^K)$ of \eqref{main_problems}, for some $0 < \alpha < 1$.
\end{theorem}

It was also known that there is a subsequential limit of $\{(u^{\epsilon}_1, \cdots, u^{\epsilon}_K)\}_{\epsilon > 0}$ as $\epsilon \rightarrow 0^+$.

\begin{theorem}
\cite[Cor 5.6]{CL2}\cite[Thm 2]{ChPaTo26}
\label{main theorem3-old}

Let $\Omega \subset \mathbb{R}^n$ be a bounded Lipschitz domain, and assume \eqref{fiassumption}-\eqref{f_i-f_j_disjointsupport} hold true. Consider the elliptic system \eqref{main_problems} with $H_R(w)(x)$ given by \eqref{H1}.
For any $\eps > 0$ and $0<R\leq 1$,  let $(u^{\eps}_1, \ldots, u^{\eps}_K)$ be a solution to \eqref{main_problems}. Then there exists a subsequence of $\{(u^{\eps}_1,\ldots, u^{\eps}_K) \}_{\eps > 0}$
that converges locally uniformly in $\Omega$ to a limit function $(u_1, \ldots, u_K)$ as $\eps \rightarrow 0^+$. Moreover, the limit function $(u_1, \ldots, u_K)$
has the following properties:
\begin{enumerate}
\item Each function $u_i$ is locally Lipschitz continuous on $\Omega$.
\item  $\Delta u_i = 0$ on $\{u_i > 0\}\cap\Omega$,  for any $i=1,\ldots, K$.
\item  For any $1 \leq i < j \leq K$, the supports of the function $u_i$ and $u_j$ are at distance at least $\sizeofboundary$ from each other. 
\end{enumerate}
\end{theorem}

\begin{definition}

The sets $\partial \{u_i > 0\} \cap \Omega, i = 1, \cdots, K$ are called the free boundaries.
\end{definition}

In addition, some preliminary properties of the free boundaries were known.

\begin{theorem}
\cite[Cor 6.2, 6.5]{CL2}\cite[Thm 7.1]{CL2}\cite[Thm 3, 4]{ChPaTo26}
\label{main theorem5-old}

If $x_0 \in \partial \{ u_i > 0\}\cap\Omega$ for some $i = 1, \ldots, K$, then there is an exterior tangent ball $B_R(y)$ at $x_0$, and the set $\{u_i > 0\}\cap\Omega$ has finite perimeter. Moreover, if $x_0 \in \partial \{u_i > 0\} \cap \Omega$, there exists a $j \neq i$ such that $\overline{B_{\sizeofboundary}(x_0)} \cap \partial \{u_j > 0\} \neq \emptyset$.  
\end{theorem}

We recall the definitions of regular and singular free boundary points.

\begin{definition}
\label{basic_definition}
\cite{CL2}\cite[Def 2.1, 2.2]{ChPaTo26_2}

For each $i = 1, \cdots, K$, we define $S_i:= \{x \in \Omega : u_i (x) > 0 \}$, and $C_i := \cup_{j \neq i} S_j$.
A free boundary point $x_0 \in \partial S_i\cap \Omega$ is said to be {\bf regular} if there exists a unique $x_1 \in \partial C_i \cap \Omega$ for which $d(x_0,x_1) = \sizeofboundary.$ A free boundary point is said to be {\bf singular} if it is not regular. 
\end{definition}

Hence, $x_0 \in \partial S_i \cap \Omega$ is singular if and only if there are two distinct points $x_1, x_2 \in \partial C_i \cap \Omega, x_1 \neq x_2$ such that $d(x_0,x_1) = d(x_0,x_2) = \sizeofboundary.$

We also fix our notation of groups.

\begin{notation}
\label{Nonation_for_Groups}
(Notation for Groups)

\begin{enumerate}

\item Denote by $GL(n;\mathbb{R})$ the group of all invertible linear maps on $\mathbb{R}^n$. 
\item Denote by $O(n;\mathbb{R})$ the group of all orthogonal linear maps on $\mathbb{R}^n$. 
\item For each $n \in \mathbb{N}$, we define $C_n$ to be the finite cyclic group of order $n$. 
\item For each $n \in \mathbb{N}$, we define $Sym(n) := \{\sigma: \{1,\cdots,n\}\rightarrow \{1,\cdots,n\}: \sigma \textnormal{ is bijective}.\}$ to be the symmetric group on $n$ letters.
\item For each $n \in \mathbb{N},$ we denote by $D_{2n} := \langle x,y| x^n = y^2 = e, yx = x^{-1}y \rangle$ the dihedral group with $2n$ elements. That is, $D_{2n}$ is the group of all rigid motions of a regular $n-$gon. 
\item For each $n \in \mathbb{N}$, we define $B_n := (C_2 \times C_2 \times \cdots \times C_2) \rtimes Sym(n)$, where $C_2 \times C_2 \times \cdots \times C_2$ denotes the direct product of $n$ copies of $C_2$. Note that $B_n$ is the hyperoctahedral group, namely the group of all isometries of a hypercube in $\mathbb{R}^n$. Hence, the group $B_n$ has order $2^nn!$.          
\item For any $n \in \mathbb{N}$, we define $J_n := \{(a_1,a_2,\cdots,a_n): a_j \in \{1,-1\} \textnormal{ for all } j = 1, \cdots, n\}$. Define 
\begin{equation}
    \alpha\beta := (\alpha_1\beta_1,\alpha_2\beta_2,\cdots,\alpha_n\beta_n) \in J_n.
\end{equation}
for any $\alpha = (\alpha_1,\alpha_2,\cdots,\alpha_n) \in J_n$ and $\beta = (\beta_1,\beta_2,\cdots,\beta_n) \in J_n$.
Note that $J_n$ is a finite abelian group of order $2^n$ under this multiplication and 
\begin{eqnarray}
    J_n \cong C_2 \times C_2 \times \cdots \times C_2, 
\end{eqnarray}
where $C_2 \times C_2 \times \cdots \times C_2$ denotes the direct product of $n$ copies of $C_2$.
\end{enumerate}
\end{notation}

We note that there is another characterization of the group $B_n$.

\begin{remark}      
\label{definition_P_and_Q}
Note that the hyperoctahedral group $B_n$ defined in Notation \ref{Nonation_for_Groups} is isomorphic to a finite subgroup of $O(n;\mathbb{R})$ as follows.

For any $k \in \{1,\cdots,n\}$, define $P_k : \mathbb{R}^n \rightarrow \mathbb{R}^n$ by
\begin{eqnarray}  
\label{definition_for_P_k}
     P_k(x_1,\cdots,x_{k-1},x_k,x_{k+1},\cdots,x_n) = (x_1,\cdots,x_{k-1},-x_k,x_{k+1},\cdots,x_n)
\end{eqnarray}
for all $(x_1,x_2,\cdots,x_n) \in \mathbb{R}^n$. Moreover, for any bijection $\sigma : \{1,\cdots,n\} \rightarrow \{1,\cdots,n\}$, we define $Q_{\sigma} : \mathbb{R}^n \rightarrow \mathbb{R}^n$ by 
\begin{eqnarray}
\label{definition_for_Q_sigma}
Q_{\sigma}(x_1,x_2,\cdots,x_n) := (x_{\sigma(1)},x_{\sigma(2)},\cdots,x_{\sigma(n)})
\end{eqnarray}
for all $(x_1,x_2,\cdots,x_n) \in \mathbb{R}^n$, so all the $P_k$ and $Q_{\sigma}$ are orthogonal transformations. Also, the subgroup of $O(n;\mathbb{R})$ generated by $\{P_k: k \in \{1,\cdots,n\}\} \cup \{Q_{\sigma}: \sigma \in Sym(n)\}$ is isomorphic to $B_n$, where $\{P_k: k \in \{1,\cdots,n\}\}$ generates the group $C_2 \times C_2 \times \cdots \times C_2$ of direct product of $n$ copies of $C_2$, and $\{Q_{\sigma}: \sigma \in Sym(n)\}$ generates the symmetric group on $n$ letters.
\end{remark}

\section{Group Theoretic Constructions of Singular Set}

In this section, we prove Theorem \ref{Main_Result_1}.
We first recall a result that will be used in the proof.
\begin{lemma}
\label{S_i_reach_the_boundary}
For any $i = 1,...,K$ and any connected component $W$ of $S_i,$ $\overline{W} \cap \{x \in \partial \Omega: f_i(x) > 0\}$ is nonempty.
\end{lemma}       

\begin{proof}
We refer the reader to the proof Lemma 8.4 in \cite{CL2} for the details.
\end{proof}

We also show that the elliptic system \eqref{main_problems} is invariant under orthogonal transformations.

\begin{lemma}
\label{orthogonal_transformation}
Let $U \subset \mathbb{R}^n$ be a bounded Lipschitz domain, and let $f_1, f_2,\cdots, f_K : (\partial U)_{\leq R} \rightarrow \mathbb{R}$ satisfies \eqref{fiassumption}-\eqref{f_i-f_j_disjointsupport}. Suppose that $(u^{\epsilon}_1,u^{\epsilon}_2,\cdots,u^{\epsilon}_K) \in C^2(U) \cap C^0(\overline{U})$ is a solution of 
\begin{equation}
\begin{cases} \Delta u^{\varepsilon}_{i}= \frac{1}{\eps^2} u^{\eps}_i  \sum_{j \neq i}  H_R(u^{\eps}_j)(x)\quad & \text{ in } U,\\
        u^{\eps}_i= f_i & \text{ on } (\partial U)_{\leq {\sizeofboundary}}, \\
        
        u^{\eps}_i \geq 0 & \text{ in } U \cup (\partial U)_{\leq \sizeofboundary},
       \end{cases}
\end{equation}
with $H_{\sizeofboundary}(w)(x)$ given by \eqref{H1}.
Let $A \in O(n;\mathbb{R})$ be such that $A(U) := \{Ax: x \in U\} = U$. Define $w_i^{\epsilon} := u^{\epsilon}_i(Ax)$ for all $\epsilon > 0, i = 1,\cdots,K$, and $x \in \Omega$. Then  
\begin{equation}
\begin{cases} \Delta w^{\varepsilon}_{i} = \frac{1}{\eps^2} w^{\eps}_i  \sum_{j \neq i}  H_R(w^{\eps}_j)(x)\quad & \text{ in } U,\\
w^{\eps}_i = f_i \circ A & \text{ on } (\partial U)_{\leq {\sizeofboundary}}, \\
w^{\eps}_i \geq 0 & \text{ in } U \cup (\partial U)_{\leq \sizeofboundary},
       \end{cases}
\end{equation}
\end{lemma}

\begin{proof}
It is clear that $w^{\eps}_i = f_i \circ A$  on  $(\partial U)_{\leq {\sizeofboundary}}$, and $w^{\eps}_i \geq 0$ in $U \cup (\partial U)_{\leq \sizeofboundary}$, so it remains to check the identity 
\begin{eqnarray}
\label{equation_for_w}
\Delta w^{\varepsilon}_{i} = \frac{1}{\eps^2} w^{\eps}_i  \sum_{j \neq i}  H_R(w^{\eps}_j)(x)
\end{eqnarray}
in $U$. To see this, we use the fact that Laplacian commutes with $A$ to obtain 
\begin{equation}
\begin{split}
\Delta w^{\varepsilon}_{i}(x) = \Delta (u^{\varepsilon}_{i}(Ax)) = (\Delta u^{\varepsilon}_{i})(Ax) = \frac{1}{\eps^2} u^{\eps}_i(Ax) \sum_{j \neq i}  H_R(u^{\eps}_j)(Ax)
\\ = \frac{1}{\eps^2} u^{\eps}_i(Ax) \sum_{j \neq i} \fint_{{B}_\sizeofboundary(Ax)}u^{\eps}_j(y)\,dy
= \frac{1}{\eps^2} u^{\eps}_i(Ax) \sum_{j \neq i}  \fint_{{B}_\sizeofboundary(x)}u^{\eps}_j(Az)\,dz = \frac{1}{\eps^2} w^{\eps}_i  \sum_{j \neq i}  H_R(w^{\eps}_j)(x) \nonumber
\end{split}
\end{equation} 
in $U$, where we made a change of variable $y = Az$. This establishes \eqref{equation_for_w}.
\end{proof}

We are now ready to provide the first example of a singular set.

\begin{theorem}
\label{Concrete_Example_of_Singular_Set_two_dim}
Consider the elliptic system \eqref{main_problems} with $H_R(w)(x)$ given by \eqref{H1}. For any $n \geq 2$ and $K \geq 3$, there exists a bounded Lipschitz domain $\Omega \subset \mathbb{R}^n$ and boundary data $f_1, \cdots, f_K : (\partial \Omega)_{\leq \sizeofboundary} \rightarrow \mathbb{R}$ satisfying \eqref{fiassumption}-\eqref{f_i-f_j_disjointsupport}, such that the singular set on the free boundary is nonempty with Hausdorff dimension exactly $n-2$. 
\end{theorem}

\begin{proof}

Fix a $K \in \mathbb{N}$ with $K \geq 3.$ Pick a $D > 0$ sufficiently large. The proof is divided into several steps. 

{\bf Step 1: Construct the bounded domain and the boundary data.}

We take $B := \{x \in \mathbb{R}^2: \norm{x} < D\}$, and take $\Omega := B \times (0,1)^{(n-2)} \subset \mathbb{R}^n$. 
We also define $E_i := \{(r,\theta) \in \mathbb{R}^2: D < r < D+\sizeofboundary, \frac{2\pi (i-1)}{K} < \theta < \frac{2\pi i}{K}\} \subset \mathbb{R}^2$, $\Omega_i := E_i \times (0,1)^{(n-2)} \subset \mathbb{R}^n$ for all $i \in \{1,2,...,K\}$.

Define $T, L: \mathbb{R}^n \rightarrow \mathbb{R}^n$ by 
\begin{equation}
\begin{split}
    T(x_1,x_2,x_3,\cdots,x_n) = (\cos(\frac{2\pi}{K})x_1 - \sin(\frac{2\pi}{K})x_2, \sin(\frac{2\pi}{K})x_1 + \cos(\frac{2\pi}{K})x_2,x_3,\cdots,x_n), \\
    L(x_1,x_2,x_3,\cdots,x_n) = (x_1,-x_2,x_3,\cdots,x_n). \nonumber
\end{split}
\end{equation}
for all $(x_1,\cdots,x_n) \in \mathbb{R}^n$. 
Observe that $T^K = L^2 = I$, (the identity map on $\mathbb{R}^n$). and $TL = T^{-1}L$. The subgroup of $GL(n;\mathbb{R})$ generated by $T$ and $L$ is isomorphic to the dihedral group $D_{2K}$.

The boundary data $f_1, f_2, ..., f_K$ are defined as follows. Choose $f_1 : (\partial \Omega)_{\leq \sizeofboundary} \rightarrow \mathbb{R}$ to be a nonzero non-negative H\"older continuous function supported in the set $\{(r,\theta) \in \mathbb{R}^2: D \leq r \leq D +\sizeofboundary, \frac{2\pi}{3K} < \theta < \frac{4\pi}{3K}\}\times (0,1)^{(n-2)} \subset \overline{E_1} \times (0,1)^{(n-2)}$ and require that $f_1$ be symmetric with respect to the hyperplane $\{(x_1,x_2,\cdots,x_n) \in \mathbb{R}^n: x_2 = \tan(\frac{\pi}{K})x_1\}.$ 

For each $i \in \{1,...,K\}$ we define $f_i : (\partial \Omega)_{\leq \sizeofboundary} \rightarrow \mathbb{R}$ by setting $f_i := f_1 \circ (T)^{-(i-1)}$. It is clear that the $K-$tuple $(f_1,f_2,...,f_K)$ of non-negative H\"older continuous functions satisfies \eqref{fiassumption}, \eqref{size_of_ball_intersect_support}, and \eqref{f_i-f_j_disjointsupport}.

{\bf Step 2: Show that if $u^{\epsilon} = (u^{\epsilon}_1, \cdots, u^{\epsilon}_K)$ is a solution of \eqref{main_problems} and $u^{\eps} \rightarrow u = (u_1, \cdots, u_K)$ along a subsequence as $\epsilon \rightarrow 0^+$, then 
\begin{enumerate}
\item $u^{\epsilon}_i(x) = u^{\epsilon}_1(T^{-(i-1)}(x))$ for all $x \in \Omega$, 
\item $u^{\epsilon}_{i}(x) = u^{\epsilon}_{i-1}(T^{2i-2}L(x))$ for all $x \in \Omega$ (here $T^{2i-2}L$ is the "reflection with respect to the hyperplane $\{(r,\theta) \in \mathbb{R}^2:\theta = \frac{2\pi}{K}(i-1)\} \times (0,1)^{(n-2)}$"), where we set $u^{\epsilon}_0 := u^{\epsilon}_K$. 
\item $S_i = T^{(i-1)}(S_1)$, where $S_i := \{u_i > 0\}$ for all $i = 1, \cdots, K$. 
\item $S_{i-1} = (T^{2i-2}L)(S_i)$, for all $i = 1, \dots, K$, where $S_0 := S_K := \{u_K > 0\}$
\end{enumerate}
}

For (1), fix a $i \in \{1,\cdots,K\}$ and define $w^{\epsilon}_j(x) := u^{\epsilon}_j(T^{-i+1}(x))$ for all $j \in \{1,\cdots,K\}$ and $x \in \Omega$. Since $T^{-i+1}$ is an orthogonal transformation, by Lemma \ref{orthogonal_transformation} we have $\Delta w^{\epsilon}_j = \frac{1}{\epsilon^2}w^{\epsilon}_j \sum_{k \neq j} H_{\sizeofboundary}(w^{\varepsilon}_k)$ in $\Omega$ and $w^{\epsilon}_j = f_j \circ T^{(-i+1)} = f_{j+i-1}$ on $(\partial \Omega)_{\leq \sizeofboundary}$. Observe that $(u^{\epsilon}_{j+i-1})_{j = 1}^{K}$ is also a solution of the same elliptic system, where we set $u^{\epsilon}_k(x) = u^{\epsilon}_j(x)$ if $k \notin \{1,2,\cdots,K\}$ and $K | (k-j)$, $j \in \{1,2,\cdots,K\}$. The result of (1) follows by applying the uniqueness part of Theorem \ref{existence_and_uniqueness} to $(u^{\epsilon}_{j+i-1})_{j = 1}^{K}$ and $(w^{\epsilon}_{j})_{j = 1}^{K}$ and set $j = 1$.

The proof of (2) proceeds similarly as in (1). Fix a $i \in \{1,\cdots,K\}$ and define $v^{\varepsilon}_j(x) = u^{\varepsilon}_j(T^{2i-2}L(x))$ for all $j \in \{1,\cdots,K\}$ and $x \in \Omega$. Since $T^{2i-2}L$ is an orthogonal transformation, by Lemma \ref{orthogonal_transformation} we have $\Delta v^{\epsilon}_j = \frac{1}{\epsilon^2}v^{\epsilon}_j \sum_{k \neq j} H_{\sizeofboundary}(v^{\varepsilon}_k)$ in $\Omega$ and $v^{\epsilon}_j = f_j \circ T^{2i-2}L$ on $(\partial \Omega)_{\leq \sizeofboundary}$. Observe that $(u^{\epsilon}_{2i-1-j})_{j = 1}^{K}$ is also a solution of the same elliptic system, where we set $u^{\epsilon}_k(x) = u^{\epsilon}_j(x)$ if $k \notin \{1,2,\cdots,K\}$ and $K | (k-j)$, $j \in \{1,2,\cdots,K\}$. The result of (1) follows by applying the uniqueness part of Theorem \ref{existence_and_uniqueness} to $(u^{\epsilon}_{2i-1-j})_{j = 1}^{K}$ and $(v^{\epsilon}_{j})_{j = 1}^{K}$ and set $j = i-1$.

For (3), we note that from (1), upon passing to a subsequence as $\epsilon \rightarrow 0 ^+$, we have $u_i(x) = u_1(T^{-(i-1)}(x))$ for all $i = 1, \cdots,K$, so 
\begin{equation}   
\begin{split}
S_i = \{x \in \Omega: u_i(x) > 0\} \\
= \{x \in \Omega: u_1(T^{-(i-1)}(x)) > 0\} = \{x \in \Omega: T^{-(i-1)}(x) \in S_1\} = T^{(i-1)}(S_1), \\
\end{split}
\end{equation}
so (3) holds. (4) is proved in a similar manner using (2).       

Therefore, there is an action of the dihedral group $\langle T,L\rangle \cong D_{2K}$ on the set $\{S_1,S_2,\cdots,S_K\}$. The action is transitive.

{\bf Step 3: We claim that $S_i \cap (\partial U_j \cap \Omega) = \emptyset$ for all $i, j = 1, \cdots, K$, where $U_j = \{(r,\theta) \in \mathbb{R}^2: 0 < r <  D, \theta \in (\frac{2\pi (j-1)}{K},\frac{2\pi j}{K})\} \times (0,1)^{(n-2)}$ for all $j \in \{1,2,\cdots,K\}$.} (We set $S_k = S_j$ and $U_k = U_j$ if $k \notin \{1,2,\cdots,K\}$ and $K | (k-j)$, $j \in \{1,2,\cdots,K\}$).

To see this, assume by contradiction that there is a $x_0 \in S_i \cap (\partial U_j \cap \Omega)$, so either $x_0 \in S_i \cap \{\theta = \frac{2\pi j}{K}\}$ or $x_0 \in S_i \cap \{\theta = \frac{2\pi (j-1)}{K}\}$. 
If $x_0 \in S_i \cap \{\theta = \frac{2\pi (j-1)}{K}\}$, then by the fact $S_{2j-1-i} = (T^{2j-2}L)(S_i),$ one has $x_0 \in S_i \cap S_{2j-1-i},$ so $S_i$ intersects $S_{2j-1-i}$, which is a contradiction. 
Similarly, in the case where $x_0 \in S_i \cap \{\theta = \frac{2\pi j}{K}\}$ we obtain a contradiction again. 

{\bf Step 4: We claim that $S_i \subset U_i$ for all $i \in \{1,2,\cdots,K\}$, where $U_i = \{(r,\theta) \in \mathbb{R}^2: 0 < r <  D, \theta \in (\frac{2\pi (i-1)}{K},\frac{2\pi i}{K})\} \times (0,1)^{n-2}$ is defined as in Step 3.}

To see this, fix a $i \in \{1,2,\cdots,K\}$, and pick a $x_0 \in S_i$. By Step 3, $x_0 \notin \partial U_j \cap \Omega$ for all $j = 1, \cdots, K$, so $x_0 \in U_k$ for some $k = 1, \cdots, K$. Assume $k \neq i$, and let $W$ be the connected component of $S_i$ with $x_0 \in \overline{W}$. By Lemma \ref{S_i_reach_the_boundary}, $\overline{W} \cap \{x \in \partial \Omega: f_i(x) > 0\}$ is nonempty, so there is a $z_0 \in W \cap U_i.$ Since $W$ is path-connected, (being an open connected set) there is a continuous map $\gamma: [0,1] \rightarrow W$ with $\gamma(0) = x_0$ and $\gamma(1) = z_0.$ The image of $\gamma$ would intersect $\partial U_i \cap \Omega.$ This implies that $S_i \cap (\partial U_i \cap \Omega) \neq \emptyset,$ which contradicts to Step 3. We have shown that $S_i \subset U_i$ for all $i \in \{1,2,...,K\}.$

{\bf Step 5: We claim that $S_i \subset V_i$ for all $i \in \{1,2,...,K\}.$ Here $V_i := \{x \in U_i: d(x,(\partial U_i \cap \Omega)) > \frac{\sizeofboundary}{2}\}$.} 

To see this, note that by the result of Step 2(3) it suffices to check this for $i = 1$. Observe that $\partial U_1 \cap \Omega = (\{(r,\theta) \in \mathbb{R}^2:r \in [0,D],\theta = \frac{2\pi}{K}\} \times (0,1)^{(n-2)}) \cup (\{(r,\theta) \in \mathbb{R}^2 :r \in [0,D],\theta = 0\} \times (0,1)^{(n-2)})$. Assume by contradiction that there is a $p_0 \in S_1 \setminus V_1$. 
By Step 4 we already have $S_1 \subset U_1,$ so $p_0 \in U_1 \setminus  V_1.$ Hence, either $d(p_0, \{(r,\theta) \in \mathbb{R}^2:r \in [0,D],\theta = 0\}\times (0,1)^{(n-2)}) < \frac{\sizeofboundary}{2}$ or $d(p_0, \{(r,\theta) \in \mathbb{R}^2:r \in [0,D],\theta = \frac{2\pi}{K}\} \times (0,1)^{(n-2)}) < \frac{\sizeofboundary}{2}.$ In the former case, we have $L(p_0) \in S_K$ and $d(p_0,L(p_0)) < \sizeofboundary$. This contradicts to the fact that $d(S_1,S_K) \geq \sizeofboundary$ (see for instance Corollary 5.6 in \cite{CL2}). Similarly, in the latter case, one would obtain $d(S_1,S_2) < \sizeofboundary,$ which is a contradiction again. We have shown that $S_i \subset V_i$ for all $i \in \{1,2,...,K\}.$

{\bf Step 6: We claim that $S_i = V_i$ for all $i \in \{1,2,...,K\}$.} 

To see this, note that by Step 2 it suffices to check this for $i = 1.$ Assume by contradiction that $S_1$ is a proper subset of $V_1,$ so there is a free boundary point $p_0 \in \partial S_1 \cap V_1$. For this $p_0$ we have $d(p_0,\partial C_1) > \sizeofboundary,$ contradicting to the fact that $d(p_0,\partial C_1) = \sizeofboundary$ (see for instance \cite{CL2} Theorem 7.1). Consequently, $S_i = V_i$ for all $i \in \{1,2,...,K\}$ and Step 6 is proved. 

{\bf Step 7: Complete the proof}

Finally, observe that the singular set of each population is the ridge of $V_i$, which is diffeomorphic to $(0,1)^{(n-2)}$, so the singular set has Hausdorff dimension exactly $n-2$.  

\end{proof}

}

We give a few remarks about Theorem \ref{Concrete_Example_of_Singular_Set_two_dim}.

\begin{remark}(About Theorem \ref{Concrete_Example_of_Singular_Set_two_dim})

\begin{enumerate}   
\item In the proof of Theorem \ref{Concrete_Example_of_Singular_Set_two_dim}, if we take $n = 2$ and $K = 4$ populations, we have the configuration shown in figure \ref{figure_for_four_populations_2D}. In this example, the free boundary in each population consists of two straight lines meeting at a singular point with angle $\frac{\pi}{2}$, and all the other points are regular. 

\item As a byproduct of the proof of Theorem \ref{Concrete_Example_of_Singular_Set_two_dim_straitification}, we see that if $\Omega$ and $f_1, \cdots, f_K$ are as in the proof of the Theorem, the free boundaries $\partial S_i \cap \Omega$ are uniquely determined. 

\item By the proof of Theorem \ref{Concrete_Example_of_Singular_Set_two_dim_straitification}, if we replace each $f_i$ by $cf_i$, for some constant $c > 0$, so that the identity $f_i := f_1 \circ (T)^{-(i-1)}$ still holds, then the free boundaries $\partial S_i \cap \Omega$ remain invariant, and the singular set remains $(n-2)$-dimensional. It remains open to construct an example in which the boundary data are not symmetric and the singular set on the free boundaries has Hausdorff dimension $n-2$.
\end{enumerate}
\end{remark}

\begin{figure}
{
\usetikzlibrary{calc}

\begin{tikzpicture}[scale=3]
    \draw[thick] (0,0) circle (1cm);
    \node at (0.9, 0.9) {$\Omega$};

    \def\hR{0.2} 

    \begin{scope}
        \clip (0,0) circle (1cm);
        
        \filldraw[fill=blue!10, draw=blue!80, thick] (\hR, \hR) rectangle (1.2, 1.2);
        
        \filldraw[fill=green!15, draw=green!80, thick] (-\hR, \hR) rectangle (-1.2, 1.2);
        
        \filldraw[fill=orange!15, draw=orange!80, thick] (-\hR, -\hR) rectangle (-1.2, -1.2);
        
        \filldraw[fill=red!10, draw=red!80, thick] (\hR, -\hR) rectangle (1.2, -1.2);
    \end{scope}

    \fill (\hR, \hR) circle (0.5pt); 
    \fill (-\hR, \hR) circle (0.5pt);
    \fill (-\hR, -\hR) circle (0.5pt);
    \fill (\hR, -\hR) circle (0.5pt);

    \node[blue!80, font=\footnotesize] at (0.55, 0.55) {$S_1$};
    \node[green!80, font=\footnotesize] at (-0.55, 0.55) {$S_2$};
    \node[orange!80, font=\footnotesize] at (-0.55, -0.55) {$S_3$};
    \node[red!80, font=\footnotesize] at (0.55, -0.55) {$S_4$};

    \draw[<->, dashed, gray] (-\hR, 0.4) -- (\hR, 0.4) node[midway, above, font=\tiny, text=black] {$R$};
    \draw[<->, dashed, gray] (0.4, -\hR) -- (0.4, \hR) node[midway, right, font=\tiny, text=black] {$R$};
\end{tikzpicture}

}
\caption{This figure illustrates Theorem \ref{Concrete_Example_of_Singular_Set_two_dim} in dimension $2$ with $4$ populations. Each population has a singular point with angle $\frac{\pi}{2}$.}
\label{figure_for_four_populations_2D}
\end{figure}
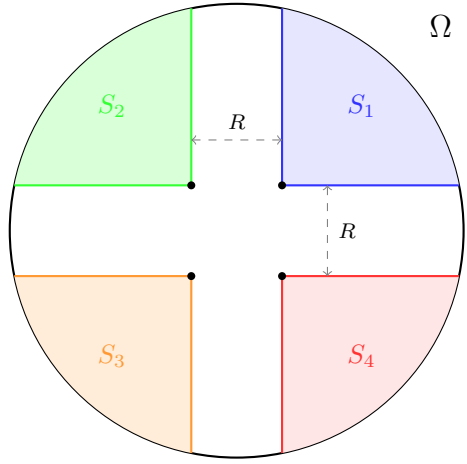

In our next example in Theorem \ref{Concrete_Example_of_Singular_Set_two_dim_straitification}, we shall construct an example in which the singular set is not a single manifold, but a union of finitely many manifolds.

\begin{theorem}
\label{Concrete_Example_of_Singular_Set_two_dim_straitification}
Consider the elliptic system \eqref{main_problems} with $H_R(w)(x)$ given by \eqref{H1}, and let $n \geq 2$. Take $\Omega := B_1(0) \subset \mathbb{R}^n$ to be the unit open ball in $\mathbb{R}^n$. Then there exists $K = 2^n$ boundary data $f_1, \cdots, f_K: (\partial \Omega)_{\leq \sizeofboundary} \rightarrow \mathbb{R}$ satisfying \eqref{fiassumption}-\eqref{f_i-f_j_disjointsupport}, such that for each $i = 1, \cdots, K$, the singular set on the free boundary takes the form 
\begin{eqnarray}
\label{stratification}
\Gamma^{\textnormal{sin}}_{i,0} \cup \Gamma^{\textnormal{sin}}_{i,1} \cup \cdots \cup \Gamma^{\textnormal{sin}}_{i,n-2},
\end{eqnarray}
where for each $k \in \{0,1,\cdots,n-2\}$, $\Gamma^{\textnormal{sin}}_{i,k}$ is nonempty and is a finite union of $k$-dimensional manifolds.  
\end{theorem}

\begin{proof}
Let $J_n$ be as in Notation \ref{Nonation_for_Groups} (7), which has cardinality $2^n$, so we may instead consider the solution $\{u^{\epsilon}_{\alpha}\}_{\alpha \in J_n}$ of the elliptic system

\begin{equation}
\label{main_problems_with_Laplace_p=1_with_new_symbol}
\begin{cases} \Delta u^{\varepsilon}_{\alpha} = \frac{1}{\eps^2}\, u^{\eps}_{\alpha}  \displaystyle\sum_{\beta \neq \alpha, \beta \in J_n}   H_{\sizeofboundary}(u^{\eps}_{\beta})(x)\quad & \text{ in } \Omega,\\
        u^{\eps}_{\alpha} = f_{\alpha} & \text{ on } (\partial \Omega)_{\leq {\sizeofboundary}}.
       \end{cases}
\end{equation}

For any $\alpha = (\alpha_1,\cdots,\alpha_n) \in J_n$, we define $p_{\alpha} := \frac{1}{\sqrt{n}}(\alpha_1,\cdots,\alpha_n) = \frac{1}{\sqrt{n}}\alpha \in \partial \Omega$. 

The proof is divided into several steps.

{\bf Step 1: Construct the boundary data $f_{\alpha}, \alpha \in J_n$.}  
Fix a $r \in (0,\frac{\sizeofboundary}{2})$. We define $f_{(1,1,\cdots,1)}$ first. We let $f_{(1,1,\cdots,1)} : (\partial \Omega)_{\leq \sizeofboundary} \rightarrow \mathbb{R}$ be supported in $(\partial \Omega)_{\leq \sizeofboundary} \cap B_r(p_{(1,1,\cdots,1)})$. Also, we construct $f_{(1,1,\cdots,1)}$ in such a way that 
\begin{eqnarray}
    f_{(1,1,\cdots,1)} \circ Q_{\sigma}(x_1,\cdots,x_n) = f_{(1,1,\cdots,1)}(x_1,\cdots,x_n)
\end{eqnarray}
for all $(x_1,x_2,\cdots,x_n) \in (\partial \Omega)_{\leq \sizeofboundary}$, where $Q_{\sigma}$ is defined as in \eqref{definition_for_Q_sigma}.       

For all $\alpha = (\alpha_1,\alpha_2,\cdots,\alpha_n) \in J_n$, we define $f_{\alpha} : (\partial \Omega)_{\leq \sizeofboundary} \rightarrow \mathbb{R}$ by setting 
\begin{eqnarray}
\begin{split}
f_{\alpha}(x_1,x_2,\cdots,x_n) = f_{(\alpha_1,\alpha_2,\cdots,\alpha_n)}(x_1,x_2,\cdots,x_n) 
\\ = f_{(1,1,\cdots,1)}(P_1^{\frac{-\alpha_1+1}{2}}P_2^{\frac{-\alpha_2+1}{2}} \cdots P_n^{\frac{-\alpha_n+1}{2}}(x_1,x_2,\cdots,x_n))
\end{split}
\end{eqnarray}
for all $(x_1,x_2,\cdots,x_n) \in (\partial \Omega)_{\leq \sizeofboundary}$ and $\sigma \in Sym(n)$, where $P_1, P_2, \cdots, P_n$ are defined as in \eqref{definition_for_P_k}. We can choose $f_{(1,1,\cdots,1)}$ such that \eqref{fiassumption}-\eqref{f_i-f_j_disjointsupport} hold.

{\bf Step 2: Show that if $u^{\epsilon} = (u^{\epsilon}_{\alpha})_{\alpha \in J_n}$ is a solution of \eqref{main_problems_with_Laplace_p=1_with_new_symbol} and $u^{\eps} \rightarrow u = (u_{\alpha})_{\alpha \in J_n}$ along a subsequence as $\epsilon \rightarrow 0^+$, then 
\begin{enumerate}
\item $u^{\epsilon}_{\alpha}(x) = u^{\epsilon}_{(1,1,\cdots,1)}(P_1^{\frac{-\alpha_1+1}{2}}P_2^{\frac{-\alpha_2+1}{2}} \cdots P_n^{\frac{-\alpha_n+1}{2}}(x))$ for all $x \in \Omega$, $\alpha = (\alpha_1,\alpha_2,\cdots,\alpha_n) \in J_n$.
\item $u^{\epsilon}_{\alpha}(x) = u^{\epsilon}_{\alpha}(x)(Q_{\sigma}(x))$ for all $x \in \Omega$ and $\sigma \in Sym(n)$.    
\item $S_{\alpha} = P_1^{\frac{-\alpha_1+1}{2}}P_2^{\frac{-\alpha_2+1}{2}} \cdots P_n^{\frac{-\alpha_n+1}{2}}(S_{(1,1,\cdots,1)})$, where $S_{\alpha} := \{u_{\alpha} > 0\}$ for all $\alpha \in J_n$.     
\item $S_{\alpha} = (Q_{\sigma})(S_{\alpha})$, for all $\alpha \in J_n$ and $\sigma \in Sym(n)$.      
\end{enumerate}}

For (1), we fix a $\alpha \in J_n$. For any $\beta \in J_n$, we define 
\begin{eqnarray}
v^{\epsilon}_{\beta}(x) := u^{\epsilon}_{\beta}(P_1^{\frac{-\alpha_1+1}{2}}P_2^{\frac{-\alpha_2+1}{2}}\cdots P_n^{\frac{-\alpha_n+1}{2}}(x)). 
\end{eqnarray}
Since $P_1^{\frac{-\alpha_1+1}{2}}P_2^{\frac{-\alpha_2+1}{2}}\cdots P_n^{\frac{-\alpha_n+1}{2}}$ is an orthogonal transformation on $\mathbb{R}^n$, by Lemma \ref{orthogonal_transformation}, we have 
\begin{equation}
\begin{cases} \Delta v^{\varepsilon}_{\beta} = \frac{1}{\eps^2}\, v^{\eps}_{\beta}  \displaystyle\sum_{\gamma \neq \beta, \gamma \in J_n}   H_{\sizeofboundary}(v^{\eps}_{\gamma})(x)\quad & \text{ in } \Omega,\\
v^{\eps}_{\beta} = f_{\beta} \circ P_1^{\frac{-\alpha_1+1}{2}}P_2^{\frac{-\alpha_2+1}{2}}\cdots P_n^{\frac{-\alpha_n+1}{2}} & \text{ on } (\partial \Omega)_{\leq {\sizeofboundary}}.
\end{cases}
\end{equation}

Observe that $(u_{\alpha\beta}^{\epsilon})_{\beta \in J_n}$ also satisfy the same elliptic system as $(v_{\beta})_{\beta \in J_n}$, by the uniqueness part of Theorem \ref{existence_and_uniqueness}, $u_{\alpha\beta}^{\epsilon}(x) = v^{\epsilon}_{\beta}(x) = u^{\epsilon}_{\beta}(P_1^{\frac{-\alpha_1+1}{2}}P_2^{\frac{-\alpha_2+1}{2}}\cdots P_n^{\frac{-\alpha_n+1}{2}}(x))$ for all $x \in \Omega$ and $\beta \in J_n$. The results follows by taking $\beta = (1,1,\cdots,1)$.

For (2), we fix a $\sigma \in Sym(n)$. For any $\beta \in J_n$, we define 
\begin{equation}
    w^{\varepsilon}_{\beta}(x) := u^{\varepsilon}_{\beta}(Q_{\sigma}(x))
\end{equation}
for all $x \in \Omega$. Since $Q_{\sigma}$ is an orthogonal transformation, by Lemma \ref{orthogonal_transformation}, we have 
\begin{equation}
\begin{cases} \Delta w^{\varepsilon}_{\beta} = \frac{1}{\eps^2}\, w^{\eps}_{\beta}  \displaystyle\sum_{\gamma \neq \beta, \gamma \in J_n}   H_{\sizeofboundary}(w^{\eps}_{\gamma})(x)\quad & \text{ in } \Omega,\\
w^{\eps}_{\beta} = f_{\beta} \circ Q_{\sigma} = f_{\beta} & \text{ on } (\partial \Omega)_{\leq {\sizeofboundary}}.
\end{cases}
\end{equation} 
However, note that $(u_{\beta}^{\epsilon})_{\beta \in J_n}$ also satisfy the same elliptic system as $(w_{\beta})_{\beta \in J_n}$, by the uniqueness part of Theorem \ref{existence_and_uniqueness}, $u_{\beta}^{\epsilon}(x) = w^{\epsilon}_{\beta}(x) = u^{\varepsilon}_{\beta}(Q_{\sigma}(x))$ for all $x \in \Omega$ and $\beta \in J_n$. This establishes (2). 

For (3), we note that by (1) and by letting $\varepsilon \rightarrow 0^+$ along a subsequence, we have $u_{\alpha}(x) = u_{(1,1,\cdots,1)}(P_1^{\frac{-\alpha_1+1}{2}}P_2^{\frac{-\alpha_2+1}{2}} \cdots P_n^{\frac{-\alpha_n+1}{2}}(x))$ for all $x \in \Omega$, $\alpha = (\alpha_1,\alpha_2,\cdots,\alpha_n) \in J_n$. Therefore, 
\begin{equation}
\begin{split}
S_{\alpha} = \{x \in \Omega: u_{\alpha}(x) > 0\} 
= \{x \in \Omega: u_{(1,1,\cdots,1)}(P_1^{\frac{-\alpha_1+1}{2}}P_2^{\frac{-\alpha_2+1}{2}} \cdots P_n^{\frac{-\alpha_n+1}{2}}(x)) > 0\}
\\ = \{x \in \Omega: P_1^{\frac{-\alpha_1+1}{2}}P_2^{\frac{-\alpha_2+1}{2}} \cdots P_n^{\frac{-\alpha_n+1}{2}}(x) \in S_{(1,1,\cdots,1)}\} 
\\ = (P_1^{\frac{-\alpha_1+1}{2}}P_2^{\frac{-\alpha_2+1}{2}} \cdots P_n^{\frac{-\alpha_n+1}{2}})^{-1}(S_{(1,1,\cdots,1)}) 
 = (P_1^{\frac{-\alpha_1+1}{2}}P_2^{\frac{-\alpha_2+1}{2}} \cdots P_n^{\frac{-\alpha_n+1}{2}})(S_{(1,1,\cdots,1)}).
\end{split}
\end{equation}   

(4) is proved similarly using (2).

{\bf Step 3: We show that $S_{\alpha} \cap (\partial U_{\beta} \cap \Omega) = \emptyset$ for all $\alpha, \beta \in J_n$, where $U_{\beta} := \{x = (x_1,x_2,\cdots,x_n) \in \Omega: (-1)^\frac{-\beta_k+1}{2}x_k > 0 \textnormal{ for all } k = 1, \cdots, n\}$ for all $\beta = (\beta_1,\cdots,\beta_n) \in J_n$}.

To show that Step 3 holds, we assume by contradiction that there exist $\alpha, \beta \in J_n$, such that $S_{\alpha} \cap (\partial U_{\beta} \cap \Omega) \neq \emptyset$. Pick a $x_0 \in S_{\alpha} \cap (\partial U_{\beta} \cap \Omega)$. Then there exists a $j \in \{1,2,\cdots,n\}$ such that $P_j(x_0) = x_0$, with $P_j$ given by Remark \ref{definition_P_and_Q}. Take $\gamma = (\gamma_1,\cdots,\gamma_n) \in J_n$ by setting $\gamma_k = \alpha_k$ for all $k \neq j$ and $\gamma_j = -\alpha_j$. It follows that $x_0 \in S_{\alpha} \cap S_{\gamma}$. This contradicts to the fact that the supports of different populations are of distance $\sizeofboundary$ from each other.  

{\bf Step 4: We show that $S_{\alpha} \subset U_{\alpha}$ for all $\alpha \in J_n$}. 

To see this, assume by contradiction that there exists a $x_0 \in S_{\alpha} \setminus U_{\alpha}$. Let $W$ be a connected component of $S_{\alpha}$ with $x_0 \in \partial W$. By Lemma \ref{S_i_reach_the_boundary}, $\overline{W} \cap \{x \in \partial \Omega: f_{\alpha}(x) > 0\} \neq \emptyset$, so there is a $z_0 \in W \cap U_{\alpha}$. Since $W$ is path-connected, (being an open connected set) there is a continuous map $\gamma: [0,1] \rightarrow W$ with $\gamma(0) = x_0$ and $\gamma(1) = z_0$. The image of $\gamma$ would intersect $\partial U_{\alpha} \cap \Omega$. This contradicts Step 3.

{\bf Step 5: We show that $S_{\alpha} \subset V_{\alpha}$ for all $\alpha \in J_n$, where $V_{\alpha} = \{x \in U_{\alpha}: d(x,\partial U_{\alpha} \cap \Omega) > \frac{\sizeofboundary}{2}\}$}.

To see this, assume by contradiction that $y_0 \in S_{\alpha} \setminus V_{\alpha}$. By Step 4, we have $y_0 \in U_{\alpha} \setminus V_{\alpha}$. Since $y_0 \in U_{\alpha} \setminus V_{\alpha}$, $d(y_0, \partial U_{\alpha}\cap \Omega) < \frac{\sizeofboundary}{2}$. Hence, there is a $j \in \{1,\cdots,n\}$ such that $d(y_0, \Omega \cap \{x_j = 0\}) < \frac{\sizeofboundary}{2}$. Take $\beta = (\beta_1,\cdots,\beta_n) \in J_n$, where $\beta_i = \alpha_i$ for all $i \neq j$ and $\beta_j = -\alpha_j$. Then $P_j(y_0) \in S_{\beta}$ and $d(y_0,P_j(y_0)) < \sizeofboundary$. This contradicts to the fact that $d(S_{\alpha},S_{\beta}) \geq \sizeofboundary$ (see for instance Corollary 5.6 in \cite{CL2}).        

{\bf Step 6: We show that $S_{\alpha} = V_{\alpha}$ for all $\alpha \in J_n$.}   

To see this, we assume by contradiction that Step 6 fails, so by Step 5, $S_{\alpha}$ is properly contained in $V_{\alpha}$. Consequently, there is a $y_0 \in \partial S_{\alpha} \cap V_{\alpha}$. For this $y_0$ we have $d(y_0, \bigcup_{\beta \in J_n, \beta \neq \alpha} (\partial S_{\beta} \cap \Omega)) > \sizeofboundary$, contradicting to the fact that $d(y_0,\bigcup_{\beta \in J_n, \beta \neq \alpha} (\partial S_{\beta} \cap \Omega)) = \sizeofboundary$ (see for instance \cite{CL2} in Theorem 7.1). Consequently, $S_{\alpha} = V_{\alpha}$ for all $\alpha \in J_n$ and Step 6 is proved.

{\bf Step 7: Conclude the result.}
By Step 6, the free boundary $\partial S_{\alpha} \cap \Omega = \partial V_{\alpha} \cap \Omega$. For any $r \in \{2,\cdots,n\}$, we let
\begin{eqnarray}
\Gamma^{\textnormal{sin}}_{\alpha,n-r} := \bigcup_{1 \leq i_i < i_2 < \cdots < i_r \leq n} [\bigcap_{j = 1}^r\{x = (x_1,x_2,\cdots,x_n) \in \overline{V_{\alpha}}: x_{i_j} = (-1)^{\frac{-\alpha_{i_j}+1}{2}}\frac{\sizeofboundary}{2}\}].
\end{eqnarray}
Then the singular set on the free boundary is $\bigcup_{r = 2}^n \Gamma^{\textnormal{sin}}_{\alpha,n-r}$, and that each $\Gamma^{\textnormal{sin}}_{\alpha,n-r}$ is a union of $C^n_r$ ("$n$ choose $r$") $(n-r)$-dimensional manifolds. We complete the proof.
\end{proof}

We give a few remarks about Theorem \ref{Concrete_Example_of_Singular_Set_two_dim_straitification}.

\begin{remark}
(About Theorem \ref{Concrete_Example_of_Singular_Set_two_dim_straitification})
\begin{enumerate}
\item In Theorem \ref{Concrete_Example_of_Singular_Set_two_dim_straitification}, we require $K = 2^n$. However, using the same technique of the proof, the result holds for some other values of $K$. For example, in the case $n = 3$, we can take $K = 4$ and take $\Omega \subset \mathbb{R}^3$ to be the interior of a regular tetrahedron. Assume $f_1, f_2, f_3, f_4 : (\partial \Omega)_{\leq \sizeofboundary} \rightarrow \mathbb{R}$ are defined near the four vertices of $\Omega$ and are symmetric, then the singular set on the free boundary also takes the form $\Gamma^{\textnormal{sin}}_{i,0} \cup \Gamma^{\textnormal{sin}}_{i,1}$, where $\Gamma^{\textnormal{sin}}_{i,1}$ is the union of three $1$-dimensional lines and $\Gamma^{\textnormal{sin}}_{i,0}$ is the point of intersection of the three lines. Similar results hold for $K = 6$ ($\Omega$ is an octahedron), $K = 12$ ($\Omega$ is a dodecahedron), $K = 20$ ($\Omega$ is an icosahedron). 
\item Also, using the same technique of proof in Theorem \ref{Concrete_Example_of_Singular_Set_two_dim} and Theorem \ref{Concrete_Example_of_Singular_Set_two_dim_straitification}, we also have the same result for the adjacent model \eqref{adjacentsegregation}. That is, for any $n \geq 2$ and $K \geq 3$, there exists a bounded Lipschitz domain $\Omega \subset \mathbb{R}^n$ and boundary data such that the singular set on the free boundary is nonempty with Hausdorff dimension exactly $n-2$. Also, for $K = 2^n$ we have a stratification result of the singular set as in \eqref{stratification}. This complements the result in \cite{CL5}, where it was shown using Almgren's frequency formula that the singular set on the free boundary has Hausdorff dimension at most $n-2$, and we show by concrete examples here that the upper bound is sharp.
\end{enumerate}
\end{remark}

{\bf Proof of Theorem \ref{Main_Result_1}:}
This follows by combining Theorem \ref{Concrete_Example_of_Singular_Set_two_dim} and Theorem \ref{Concrete_Example_of_Singular_Set_two_dim_straitification}. 
$\Box$

\begin{remark}
Our main result Theorem \ref{Main_Result_1} shows that the singular set of dimension $(n-2)$ exists. It remains an open problem that if the Hausdorff dimension of the singular set is always at most $(n-2)$.    
\end{remark}

Finally, we give an example of two populations on concentric spheres in $\mathbb{R}^n$ in which the free boundaries consist of regular points only.

\begin{example}
\label{Example_with_no_singular_points}
Fix $D_1 > \sizeofboundary$ and $D_2 > \sizeofboundary$, and take $\Omega = \{x \in \mathbb{R}^n: D_1 < \norm{x} < D_1 + D_2\}$. Define $f_1, f_2 : (\partial \Omega)_{\leq \sizeofboundary} \rightarrow \mathbb{R}$ by setting $f_1 = 1$ on $\{x \in \mathbb{R}^n: D_1-\sizeofboundary \leq \norm{x} \leq D_1\}; $ $f_1 = 0$ on $\{x \in \mathbb{R}^n: D_1+D_2 \leq \norm{x} \leq D_1+D_2+\sizeofboundary\}$, and define $f_2 = 1-f_1$ on $(\partial \Omega)_{\leq \sizeofboundary}$. Then there exists a $D^* \in (D_1,D_1+D_2-\sizeofboundary)$ such that $\partial S_1 \cap \Omega = \{x \in \mathbb{R}^n: \norm{x} = D^*\}$ and $\partial S_2 \cap \Omega = \{x \in \mathbb{R}^n: \norm{x} = D^*+\sizeofboundary\}$, and all the free boundary points are regular.
\end{example}
\begin{proof}   
Let $(u^{\epsilon}_1, u^{\epsilon}_2) \in C^{2,\alpha}(\Omega;\mathbb{R}^2) \cap C^{\alpha}(\overline{{\Omega}};\mathbb{R}^2)$ be a solution of \eqref{main_problems} with $H_R(w)(x)$ given by \eqref{H1}. We proceed in two steps.  

{\bf Step 1: We show that for any $A \in O(n;\mathbb{R})$, we have $u^{\epsilon}_i(Ax) = u^{\epsilon}_i(x)$ for all $x \in \Omega$ and $i = 1, 2$.}

To see this, we fix a $A \in O(n;\mathbb{R})$, and define $w^{\epsilon}_i = u^{\epsilon}_i \circ A$ for all $i = 1, 2$. Since $A$ is an orthogonal transformation, by Lemma \ref{orthogonal_transformation}, $(w^{\epsilon}_1, w^{\epsilon}_2) \in C^{2,\alpha}(\Omega;\mathbb{R}^2) \cap C^{\alpha}(\overline{{\Omega}};\mathbb{R}^2)$ is also a solution of \eqref{main_problems} with $H_R(w)(x)$ given by \eqref{H1}. However, by Theorem \ref{existence_and_uniqueness}, the solution of \eqref{main_problems} is unique, so $(u^{\epsilon}_1, u^{\epsilon}_2) = (w^{\epsilon}_1, w^{\epsilon}_2) = (u^{\epsilon}_1 \circ A, u^{\epsilon}_2 \circ A)$. 

{\bf Step 2: Conclude the result.}

By passing to the limit as $\epsilon \rightarrow 0^+$ along a subsequence in Step 1, we have $u_i(Ax) = u_i(x)$ for all $x \in \Omega$ and $i = 1, 2$. Therefore, for any $x, y \in \Omega$ with $\norm{x} = \norm{y}$, we have $u_i(x) > 0$ if and only if $u_i(y) > 0$. Hence, there is a $D^*$ such that $\partial S_1 \cap \Omega = \{x \in \mathbb{R}^n: \norm{x} = D^*\}$. Since $d(\partial S_1 \cap \Omega, \partial S_2 \cap \Omega) = \sizeofboundary$ (see for instance Theorem \ref{main theorem5-old}), $\partial S_2 \cap \Omega = \{x \in \mathbb{R}^n: \norm{x} = D^*+\sizeofboundary\}$ and $D^* \in (D_1,D_1+D_2-\sizeofboundary)$. Also, all the free boundary points are regular.
\end{proof}

We give a remark about Example \ref{Example_with_no_singular_points}.

\begin{remark}

In Example \ref{Example_with_no_singular_points}, there is an action of the multiplicative group $S^1 := \{z \in \mathbb{C}: |z| = 1\}$ on $\{S_1,S_2\}$. However, unlike the examples in Theorem \ref{Concrete_Example_of_Singular_Set_two_dim} and Theorem \ref{Concrete_Example_of_Singular_Set_two_dim_straitification}, (where we have the action from {\bf finite discrete} groups $D_{2n}$ and $B_n$), in Example \ref{Example_with_no_singular_points} the group $S^1$ is a {\bf continuous} group. This is reflected by the fact that singular free boundary points exist in Theorem \ref{Concrete_Example_of_Singular_Set_two_dim} and Theorem \ref{Concrete_Example_of_Singular_Set_two_dim_straitification}, but not in Example \ref{Example_with_no_singular_points}. 
\end{remark}

\section*{Acknowledgment}
The author would like to express his gratitude to Professor Monica Torres and Stefania Patrizi for the fruitful discussions during the preparation of the paper.

\end{document}